\documentclass[12pt]{article}

\usepackage{graphicx}
\usepackage{amsthm,amsfonts,amssymb,latexsym,amsmath}
\usepackage{hyperref}

\newtheorem{theorem}{Theorem}
\newtheorem{lemma}{Lemma}

\newtheorem{proposition}{Proposition}

\newcommand{\be}{\begin{eqnarray}}
\newcommand{\ee}{\end{eqnarray}}
\newcommand{\nn}{{\nonumber}}
\newcommand{\Z}{{\mathbb Z}}

\newcommand{\sF}{{\cal F}}
\newcommand{\sL}{{\cal L}}
\newcommand{\Mo}{{\rm Mo}}

\begin{document}
\title{The Mostar index of Fibonacci and Lucas cubes}
\author{
	\"Omer E\u{g}ecio\u{g}lu
	\thanks{Department of Computer Science, University of California Santa Barbara, Santa Barbara, California 93106, USA. email: omer@cs.ucsb.edu}
	\quad
	Elif Sayg{\i}
	\thanks{Department of Mathematics and Science Education, Hacettepe University, 06800, Ankara, Turkey. email: esaygi@hacettepe.edu.tr}
	\quad
	Z\"ulf\"ukar Sayg{\i}
	\thanks{Department of Mathematics, TOBB University of Economics and Technology, 06560, Ankara, Turkey. email: zsaygi@etu.edu.tr}
}


\maketitle

\begin{abstract}
The Mostar index of a graph was defined by
Do\v{s}li\'{c}, Martinjak, \v{S}krekovski, Tipuri\'{c} Spu\v{z}evi\'{c} and Zubac
in the context of the study of the properties of chemical graphs.
It measures how far a given graph is from being distance-balanced.
In this paper, we determine the 
Mostar index of two well-known families of graphs: Fibonacci cubes and Lucas cubes. 
\end{abstract}

\maketitle

\section{Introduction}\label{section.intro}
We consider 
what is termed the {\em Mostar index} of Fibonacci and Lucas cubes. 
These two families of graphs 
are special subgraphs of hypercube graphs.
They were introduced as alternative interconnection networks to 
hypercubes, and have been studied extensively because of their 
interesting graph theoretic properties. 
The Mostar index of a graph was introduced in \cite{Mostar1}.

Let $G = (V(G),E(G))$ be a graph with vertex set $V(G)$ and edge set $E(G)$. 
For any $uv\in E(G)$ let $n_u(G)$ denote the number of vertices in $V(G)$ 
that are closer (w.r.t. the standard shortest path metric) to $u$ than to $v$; 
and let $n_v(G)$ denote the number of vertices in $V(G)$  that are closer to $v$ than to $u$. 
The
{\em Mostar index} of $G$ is defined in \cite{Mostar1} as
$$\Mo(G) = \sum_{uv\in E(G)}|n_u(G)-n_v(G)|~.$$
When
$G$ is clear from the context we will write
$n_u = n_u(G)$ and $n_v = n_v(G).$

Distance related properties of graphs such 
as the Wiener index, irregularity and Mostar index have been studied for various families of graphs 
in the literature. 

The Wiener index $W(G)$ of a connected graph $G$ is defined as the sum of 
distances over all unordered pairs of vertices of $G$. It is 
determined for Fibonacci cubes and Lucas cubes in \cite{Wiener.Fibo}. The irregularity 
of a graph is another distance invariant measuring how much the graph differs 
from a regular graph 
and  Albertson index (irregularity) is defined as the sum of $|deg(u)-deg(v)|$ over all 
edges $uv$ in the graph \cite{Albertson}. 
The irregularity of Fibonacci cubes and Lucas cubes are studied in 
\cite{Klavzar_irr,FiboLucas_irr}. The relation between 
the Mostar index and the irregularity of graphs and their difference are investigated 
in \cite{Mostar3}. Recently, the Mostar index of trees and product graphs have been investigated in \cite{Mostar2}.

In this work, we determine the Mostar index of Fibonacci cubes and Lucas cubes. 
As a consequence 
we derive a relation
between the Mostar and the Wiener indices for Fibonacci cubes, giving
an alternate expression to the closed formula for $W(\Gamma_n)$  calculated in \cite{Wiener.Fibo}.

\section{Preliminaries}\label{section.prelim}
We use the notation $[n]=\{1,2, \ldots,n\}$ for any $n\in\Z^+$. Let $B=\{0,1\}$ and 
$$B_n=\{b_1 b_2 \ldots b_n\mid b_i\in B,\ \forall i\in[n]\}$$
denote the set of all binary strings of length $n$. Special subsets of $B_n$ defined as
$$\sF_n=\{b_1 b_2 \ldots b_n\mid b_i\cdot b_{i+1}=0,\ \forall i\in[n-1]\}$$
and
$$\sL_n=\{b_1b_2 \ldots b_n\mid b_i\cdot b_{i+1}=0,\ \forall i\in[n-1]\mbox{ and } b_1\cdot b_{n}=0 \}$$
are the set of all Fibonacci strings and Lucas strings of length $n$ respectively.

The $n$-dimensional hypercube $Q_n$ has vertex set $B_n$. 
Two vertices are adjacent if and only if they differ in exactly one 
coordinate in their string representation. 
For $n\ge 1$ the Fibonacci cube $\Gamma_n$ and the Lucas cube $\Lambda_n$ are defined 
as the subgraphs of $Q_n$ induced by the Fibonacci strings 
$\sF_n$ and Lucas strings $\sL_n$ of length 
$n$ \cite{Hsu,Lucas}. For convenience, we take $\Gamma_0 = K_1$ whose only 
vertex is represented by the empty string. 
 
One can classify the binary strings defining the vertices of $\Gamma_n$ by 
whether or not $b_1= 0$ or $b_1=1$. In this way $\Gamma_n$ decomposes into a 
subgraph $\Gamma_{n-1}$ whose vertices start with 0 and a subgraph $\Gamma_{n-2}$ 
whose vertices start with 10 in $\Gamma_n$. This decomposition can be denoted by
\be
\Gamma_n= 0\Gamma_{n-1}+10\Gamma_{n-2}~.
\nn\ee
Furthermore, $0\Gamma_{n-1}$ in turn has a subgraph $00\Gamma_{n-2}$ and  
there is a perfect matching between $00\Gamma_{n-2}$ and $10\Gamma_{n-2}$, 
whose edges are called {\em link edges}. 
This decomposition is the {\em fundamental decomposition} of $\Gamma_n$. In a similar way we can also decompose $\Gamma_n$ as
\be
\Gamma_n= \Gamma_{n-1}0+\Gamma_{n-2}01~.
\nn\ee
We refer to \cite{survey} for further details on $\Gamma_n$ .

For $n\ge 2$, $\Lambda_n$ is obtained from $\Gamma_n$ by 
deleting the vertices that start and end with 1. This gives 
the fundamental decomposition of $\Lambda_n$ as 
\be
\Lambda_n= 0\Gamma_{n-1}+10\Gamma_{n-3}0~.
\nn\ee
Here
$0\Gamma_{n-1}$ has a subgraph $00\Gamma_{n-3}0$ and there is a 
perfect matching between $00\Gamma_{n-3}0$ and $10\Gamma_{n-3}0$.

Fibonacci numbers $f_n$ are defined by the recursion 
$f_n = f_{n-1} + f_{n-2}$ for $n\ge 2$, with $f_0 = 0$ and $f_1 = 1$. 
Similarly, the Lucas numbers $L_n$ are defined
by $L_n = L_{n-1} + L_{n-2}$ for $n\ge 2$, with $L_0 = 2$ and $L_1 = 1$.
It is well-known that $|V(Q_n)|=|B_n|=2^n$, 
$|V(\Gamma_n)|=|\sF_n|=f_{n+2}$ and $|V(\Lambda_n)|=|\sL_n|=L_{n}$. 

For any binary string $s$ we let $w_H(s)$ denote 
the Hamming weight of $s$, that is, the number of 
its nonzero coordinates. The {\em XOR} of two binary strings $s_1$ and $s_2$ of length 
$n$, denoted by $s_1 \oplus s_2$ is defined as the string of length $n$ 
whose coordinates are the modulo 2 sum of the coordinates of $s_1$ and $s_2$. 
The distance $d(u,v)$ between two vertices $u$ and $v$ 
of the hypercube, the Fibonacci cube and the Lucas cube
is equal to the Hamming distance between the string representations of $u$ and $v$.
In other words, 
$d(u,v)=d_H(s_1,s_2)=w_H(s_1 \oplus s_2)$ for any of these graphs,
by assuming $u$ and $v$ have string representations $s_1$ and $s_2$, respectively. 
    
\section{The Mostar index of Fibonacci cubes}
For any $uv\in E(\Gamma_{n})$ let the string representations of $u$ and $v$ be 
$u_1 u_2\ldots u_n$ and $v_1 v_2\ldots v_n$, respectively. By the structure of 
$\Gamma_{n}$ we know that $d(u,v)=1$, that is, there is only one 
index $k$ for which $u_k\neq v_k$. 
\begin{lemma}\label{lemma1}
	For $n\ge 2$, assume that $uv\in E(\Gamma_n)$ with $u_k=0$ and $v_k=1$ for some 
$k\in [n]$. Then $n_u(\Gamma_n)=f_{k+1}f_{n-k+2}$ and $n_v(\Gamma_n)=f_{k}f_{n-k+1}$.
\end{lemma}
\begin{proof}
The result is clear for $n=2$. Assume that $n\ge 3$, $1<k<n$ and let 
$\alpha\in V(\Gamma_{n})$ have string representation $b_1 b_2 \ldots b_n$. 
Since $uv\in E(\Gamma_n)$, $u$ and $v$ must be of the form 
$a_1\ldots a_{k-1}0a_{k+1}\ldots a_n$ and $a_1\ldots a_{k-1}1a_{k+1}\ldots a_n$, respectively. 
Since
$v\in V(\Gamma_n)$ we must have $a_{k-1}=a_{k+1}=0$.
From these representations we observe that the difference between 
$d(\alpha,u)$ and $d(\alpha,v)$ depends on the value of $b_k$ only. 
If $b_k=0$ we have $d(\alpha,u)=d(\alpha,v)-1$ and if $b_k=1$ we have 
$d(\alpha,u)=d(\alpha,v)+1$. Therefore, 
the vertices whose $k$th coordinate is 0 are closer to $u$ than $v$; and 
the vertices whose $k$th coordinate is 1 are closer to $v$ than $u$. 
Hence $n_u(\Gamma_n)$ is equal to the number of vertices in $\Gamma_n$ 
whose $k$th coordinate is 0. 
These vertices have string representation of the form $\beta_10\beta_2$ where $\beta_1$ is 
a Fibonacci string of length $k-1$ and $\beta_2$ is a Fibonacci string of length $n-k$. 
Consequently
$n_u(\Gamma_n)=f_{k+1}f_{n-k+2}$. Similarly $n_u(\Gamma_n)$ 
is number of vertices of the form $\beta_3010\beta_4$, and this is equal to $f_{k}f_{n-k+1}$.

For the case $k=1$ we have $u\in V(0\Gamma_{n-1})$ and $v\in V(10\Gamma_{n-2})$. Then 
$n_u(\Gamma_n)=|V(0\Gamma_{n-1})|=f_{n+1}$ and $n_v(\Gamma_n)=|V(10\Gamma_{n-2})|=f_{n}$. 
Similarly, for $k=n$ we have $u\in V(\Gamma_{n-1}0)$ and $v\in V(\Gamma_{n-2}01)$. This 
gives again $n_u(\Gamma_n)=f_{n+1}$ and $n_v(\Gamma_n)=f_{n}$ for $k=n$. As $f_1=f_2=1$, these 
are also of the form claimed.
\end{proof}

To find the Mostar index of Fibonacci cubes we only need to find the number of 
edges $uv$ in $\Gamma_n$ for which $u_k=0$ and $v_k=1$ for a 
fixed $k\in [n]$ and add up these contributions over $k$.
\begin{lemma}\label{lemma2}
	For $n\ge 2$, assume that $uv\in E(\Gamma_n)$ with $u_k=0$ and $v_k=1$ for some 
$k\in [n]$. Then the number of such edges in $\Gamma_n$ is equal to $f_{k}f_{n-k+1}$.
\end{lemma}
\begin{proof}
	As in the proof of Lemma \ref{lemma1} the result is clear for $n=2$. 
Assume that $n\ge 3$. For $1<k<n$ we know that $u$ and $v$ are of the 
form $a_1\ldots a_{k-2}000a_{k+2}\ldots a_n$ and $a_1\ldots a_{k-2}010a_{k+2}\ldots a_n$. 
Then the number edges $uv$ in $\Gamma_n$ satisfying $u_k=0$ and $v_k=1$ is equal to 
the number of vertices of the form $a_1\ldots a_{k-2}000a_{k+2}\ldots a_n$, 
which gives the desired result. 
	
For the boundary cases
$k=1$ and $k=n$ we need to find the number of 
vertices of the form $00a_3\ldots a_n$ and $a_1\ldots a_{n-2}00$, respectively. Clearly, 
this number is equal to $|V(00\Gamma_{n-2})|=f_n$ and $f_1=1$. This completes the proof.
\end{proof}  

Using Lemma \ref{lemma1} and Lemma \ref{lemma2} we obtain the following main result.
\begin{theorem}\label{thm.main}
The Mostar index of Fibonacci cube $\Gamma_n$ is given by
\begin{equation}\label{formula1}
\Mo(\Gamma_n) = 
	 \sum_{k=1}^{n} f_{k}f_{n-k+1}\left(f_{k+1}f_{n-k+2}-f_{k}f_{n-k+1}\right) ~.
\end{equation}
\end{theorem}
\begin{proof}
Let $uv\in E(\Gamma_n)$ with $u_k=0$ and $v_k=1$ for some $k\in [n]$. Then from Lemma \ref{lemma1} we know that 
$$|n_u-n_v|=f_{k+1}f_{n-k+2}-f_{k}f_{n-k+1} $$ 
and therefore using  Lemma \ref{lemma2} we have
\be
	\Mo(\Gamma_n) &=& \sum_{uv\in E(\Gamma_n)}|n_u-n_v| \nn\\
	&=& \sum_{k=1}^{n} f_{k}f_{n-k+1}\left(f_{k+1}f_{n-k+2}-f_{k}f_{n-k+1}\right) 
~.
\nn\ee
\end{proof}
Note that 
$f_{k+1}f_{n-k+2}-f_{k}f_{n-k+1}= f_{k}f_{n-k}+f_{k-1}f_{n-k+2}$ so that we can equivalently 
write 
\begin{equation}\label{formula2}
\Mo(\Gamma_n) = 
\sum_{k=1}^{n} f_{k}f_{n-k+1}\left(f_{k}f_{n-k}+f_{k-1}f_{n-k+2}\right)~.
\nn\end{equation}

In Section \ref{sec.rec} Theorem \ref{thm.main3},
we present a closed form formula for $\Mo(\Gamma_n)$ obtained by using the 
theory of generating functions. 

Next we consider the Mostar index of Lucas cubes.

\section{The Mostar index of Lucas cubes}
We know that $\Lambda_{2}=\Gamma_2$ and therefore
$\Mo(\Gamma_2)=\Mo(\Lambda_2)=2$. 

For any $uv\in E(\Lambda_{n})$ let the string representations of $u$ 
and $v$ be  $u_1 u_2\ldots u_n$ and $v_1v_2\ldots v_n$, respectively. 
We know that $d(u,v)=1$ and there is only one index $k$ for which $u_k\neq v_k$. 
Similar to Lemma \ref{lemma1} and Lemma \ref{lemma2}
we have the following result.
\begin{lemma}\label{lemma3}
	For $n\ge 3$, assume that $uv\in E(\Lambda_n)$ with $u_k=0$ and $v_k=1$ for some 
$k\in [n]$. Then $n_u(\Lambda_n)=f_{n+1}$ and $n_v(\Lambda_n)=f_{n-1}$.
\end{lemma}
\begin{proof}
	Assume that $1<k<n$ and let $\alpha\in V(\Lambda_n)$ having string representation $b_1b_2\ldots b_n$. 
	Since $uv\in E(\Lambda_n)$, 
$u$ and $v$ must be of the form $a_1\ldots a_{k-2}000a_{k+2}\ldots a_n$ and 
$a_1\ldots a_{k-2}010a_{k+2}\ldots a_n$, respectively. Then, if $b_k=0$ we have $d(\alpha,u)=d(\alpha,v)-1$ and if $b_k=1$ we have $d(\alpha,u)=d(\alpha,v)+1$. Therefore, $n_u(\Lambda_n)$ and $n_v(\Lambda_n)$  are equal to the number of vertices in $\Lambda_n$ 
whose $k$th coordinate is 0 and 1, respectively. 
Therefore we need to count the number of Lucas strings of the form $\beta_10\beta_2$ and $\beta_3010\beta_4$ which gives $n_u(\Lambda_n)=f_{n+1}$ and $n_v(\Lambda_n)=f_{n-1}$.
	
	For the case $k=1$, using the fundamental decomposition of $\Lambda_n$ we have $u\in V(0\Gamma_{n-1})$ and $v\in V(10\Gamma_{n-3}0)$. Then 
	$n_u(\Lambda_n)=|V(0\Lambda_n)|=f_{n+1}$ and $n_v(\Lambda_n)=|V(10\Gamma_{n-3}0)|=f_{n-1}$. 
	Similarly, for $k=n$ we have the same results $n_u(\Lambda_n)=f_{n+1}$ and $n_v(\Lambda_n)=f_{n-1}$.
	  
\end{proof}
For any $uv\in E(\Lambda_{n})$ using Lemma \ref{lemma3} we have 
$$|n_u(\Lambda_n)-n_v(\Lambda_n)|=f_{n+1}-f_{n-1}=f_n~.$$
Since the number of edges in $\Lambda_n$ is $nf_{n-1}$ \cite{Lucas}, similar to the Theorem \ref{thm.main} we have the following result .
\begin{theorem}\label{thm.main2}
	The Mostar index of Lucas cube $\Lambda_n$ is 
given by
	$$\Mo ( \Lambda_n) = n f_n f_{n-1}~.$$
\end{theorem}
Here we remark that the vertices of Lucas cubes are represented 
by Lucas strings which are circular binary strings that avoid the pattern ``11".
Because of this symmetry, the derivation of a
closed formula of Theorem \ref{thm.main2} 
for the Mostar index of Lucas 
cube $\Lambda_n$ is easier than the one for $\Gamma_n$, in which 
the first and the last coordinates behave differently from the others. 
 
\section{A closed formula for $\Mo(\Gamma_n)$}\label{sec.rec}
By the fundamental decomposition of $\Gamma_n$, the set 
of edges $E(\Gamma_n)$ consists of three distinct types:
\begin{enumerate}
\item The edges in $0\Gamma_{n-1}$, which we  denote by $E(0\Gamma_{n-1})$.
\item The link edges between $10\Gamma_{n-2}$ and $00\Gamma_{n-2}\subset 0\Gamma_{n-1}$, which we
denote by $C_{n}$.
\item The edges in $10\Gamma_{n-2}$, which we denote by $E(10\Gamma_{n-2})$~.
\end{enumerate}
In other words we have the partition
$$E(\Gamma_n)=E(0\Gamma_{n-1})
\cup C_{n} \cup E(10\Gamma_{n-2}) ~.$$
We keep track of the contribution of each part of this decomposition by setting
for $n\ge 2$,
\be\label{defn.Mn}
M_n(x,y,z)=\sum_{uv\in E(0\Gamma_{n-1})}|n_u-n_v|x+
\sum_{uv\in C_{n}}|n_u-n_v|y +
\sum_{uv\in E(10\Gamma_{n-2})}|n_u-n_v|z ~.
\ee 
Clearly, $\Mo(\Gamma_{n})=M_n(1,1,1)$.
By direct inspection we observe that
\begin{eqnarray*}
	M_2&=&x+y \\
	M_3&=&4x+2y+z \\
	M_4&=&16x+6y+6z \\
	M_5&=&54x+15y+23z 
\end{eqnarray*}
which gives 
\begin{eqnarray*}
	\Mo(\Gamma_{2})&=&M_2(1,1,1)=2 \\
	\Mo(\Gamma_{3})&=&M_3(1,1,1)=7 \\
	\Mo(\Gamma_{4})&=&M_4(1,1,1)=28 \\
	\Mo(\Gamma_{5})&=&M_5(1,1,1)=92~, 
\end{eqnarray*}
consistent with the values that are calculated using Theorem \ref{thm.main}.
 
By using the fundamental decomposition of $\Gamma_n$ we obtain the following useful result.
\begin{proposition}\label{prop.main.rec}
 For $n\ge 2$ the polynomial $M_n(x,y,z)$ satisfies 
 $$M_n(x,y,z)=M_{n-1}(x+z,0,x)+M_{n-2}(2x+z,x+z,x+z)+f_{n-1}\left(f_{n}+f_{n-2}\right)x+f_nf_{n-1}y$$
 where $M_0(x,y,z)=M_1(x,y,z)=0$.  
\end{proposition}
\begin{proof}
By the definition (\ref{defn.Mn}), there are three cases to consider:
\begin{enumerate}
\item Assume that $uv\in C_{n}$ such that $u\in V(0\Gamma_{n-1})$ and $v\in V(10\Gamma_{n-2})$:\\
We know that $d(u,v)=1$ and the string representations of $u$ and $v$ 
must be of the form $00b_3\ldots b_n$ and $10b_3\ldots b_n$, respectively. 
Then using Lemma \ref{lemma1} with $k=1$ we have $|n_u-n_v|=f_{n+1}-f_n=f_{n-1}$ 
for each edge $uv$ in $C_{n}$. As $|C_n|=f_n$ all of these edges contribute 
$f_nf_{n-1}y$ to $M_n(x,y,z)$.
	
	
\item Assume that $uv\in E(10\Gamma_{n-2})$:\\  
Let the string representations of $u$ and $v$ be
$10u_3\ldots u_n$ and $10v_3\ldots v_n$, respectively. Using the fundamental 
decomposition of $\Gamma_n$ there exist vertices of the form 
$u'=0u_3\ldots u_n$ and $v'=0v_3\ldots v_n$ in $V(\Gamma_{n-1})$; 
$u''=u_3\ldots u_n$ and $v''=v_3\ldots v_n$ in $V(\Gamma_{n-2})$. 
Then $n_u$ counts the number of vertices $0\alpha\in V(0\Gamma_{n-1})$ 
and $10\beta \in V(10\Gamma_{n-2})$ satisfying $d(0\alpha,u)<d(0\alpha,v)$ 
and $d(10\beta,u)<d(10\beta,v)$. For any $0\alpha\in V(0\Gamma_{n-1})$ 
we know that $d(0\alpha,u)=d(\alpha,u')+1$ and $d(0\alpha,v)=d(\alpha,0v')+1$.
Therefore, for a fixed $0\alpha\in V(0\Gamma_{n-1})$, 
$d(\alpha,u')<d(\alpha,v')$ if and only if $d(0\alpha,u)<d(0\alpha,v)$. 
Similarly, for any $10\beta\in V(10\Gamma_{n-2})$ we have $d(10\beta,u)=d(\beta,u'')$ 
and $d(\beta,v)=d(\beta,v'')$.
Then we can write 
\be\label{eq.case.2}
\sum_{uv\in E(10\Gamma_{n-2})}\big|n_u(\Gamma_n)-n_v(\Gamma_n)\big|
&=&\sum_{u'v'\in E(\Gamma_{n-1})}\big|n_{u'}(\Gamma_{n-1})-n_{v'}(\Gamma_{n-1})\big| \nn\\
& &+\sum_{u''v''\in E(\Gamma_{n-2})}\big|n_{u''}(\Gamma_{n-2})-n_{v''}(\Gamma_{n-2})\big|~.
\nn\ee 
Note that $\Gamma_{n-1}=0\Gamma_{n-2}+10\Gamma_{n-3}$ and 
the edge $u'v'\in E(\Gamma_{n-1})$ is an edge in the set $E(0\Gamma_{n-2})$. 
Furthermore $u''v''\in E(\Gamma_{n-2})$ is an arbitrary edge. Then by the 
definition \eqref{defn.Mn} of $M_n$ we have 
$$\sum_{u'v'\in E(\Gamma_{n-1})}\big|n_{u'}(\Gamma_{n-1})-n_{v'}(\Gamma_{n-1})\big|=
M_{n-1}(1,0,0)$$
and 
$$\sum_{u''v''\in E(\Gamma_{n-2})}\big|n_{u''}(\Gamma_{n-2})-n_{v''}(\Gamma_{n-2})\big|=
M_{n-2}(1,1,1)~.$$
Hence all of these edges $uv\in E(10\Gamma_{n-2})$ contribute 
$\big(M_{n-1}(1,0,0)+M_{n-2}(1,1,1)\big)z$ to $M_n(x,y,z)$.
		
\item Assume that $uv\in E(0\Gamma_{n-1})$: \\
Since $0\Gamma_{n-1}=00\Gamma_{n-2}+010\Gamma_{n-3}$ we have three subcases to consider here.
\begin{enumerate}
\item Assume that $uv\in C_{n-1}$ such that $u\in00\Gamma_{n-2}$ 
and $v\in 010\Gamma_{n-3}$ .\\
Then using Lemma \ref{lemma1} with $k=2$ we have 
$|n_u-n_v|=f_3f_{n}-f_2f_{n-1}=2f_n-f_{n-1}=f_n+f_{n-2}$ 
for each edge $uv$ in $C_{n}$. As $|C_{n-1}|=f_{n-1}$ all of 
these edges contribute $f_{n-1}(f_n+f_{n-2})x$ to $M_n(x,y,z)$.

\item Assume that $uv\in E(010\Gamma_{n-3})$:\\
Let the string representations of $u$ and $v$ are of the form 
$010u_4\ldots u_n$ and $010v_4\ldots v_n$ respectively. Using the fundamental 
decomposition of $\Gamma_n$ there exist vertices of the form 
$u'=000u_4\ldots u_n$ and $v'=000v_4\ldots v_n$ in $V(0\Gamma_{n-1})$; 
$u''=0u_4\ldots u_n$ and $v''=0v_4\ldots v_n$ in $V(\Gamma_{n-2})$.
Then for any  $10\alpha\in V(10\Gamma_{n-2})$ 
we know that $d(10\alpha,u)=d(10\alpha,u')+1=d(\alpha,u'')+2$ and we know that $d(10\alpha,v)=d(10\alpha,v')+1=d(\alpha,v'')+2$. Therefore for all $10\alpha\in V(10\Gamma_{n-2})$ we count their total contribution to $M_n$ by $M_{n-2}(1,0,0)x$ in this case. Furthermore, as  $uv\in E(010\Gamma_{n-3})$ we have $uv\in E(0\Gamma_{n-1})$, and for all $0\alpha\in V(0\Gamma_{n-1})$ we count their total contribution to $M_n$ by $M_{n-1}(0,0,1)x$ by using the definition of $M_{n-1}$.
Hence, the edges  $uv\in E(010\Gamma_{n-3})$ contribute $\big(M_{n-1}(0,0,1)+M_{n-2}(1,0,0)\big)x$ to $M_n(x,y,z)$. 

\item Assume that $uv\in E(00\Gamma_{n-2})$.\\
These edges are the ones of $E(0\Gamma_{n-1})$ that are not in 
$E(010\Gamma_{n-3})$ and $C_{n-1}$ (not created 
during the connection of $00\Gamma_{n-2}$ and $010\Gamma_{n-3}$). 
Then similar to the Case 2 and using the definition \eqref{defn.Mn} of $M_n$ 
these edges contribute $\big(M_{n-1}(1,0,0)+M_{n-2}(1,1,1)\big)x$ to $M_n(x,y,z)$.

\end{enumerate}
\end{enumerate}
Combining all of the above cases and noting 
$M_{n-1}(0,0,1)x=M_{n-1}(0,0,x)$, $M_{n-2}(1,0,0)x=M_{n-2}(x,0,0)$, 
$M_{n-2}(1,1,1)x=M_{n-2}(x,x,x)$ we complete the proof. 
\end{proof}


If we write $M_n(x,y,z)=a_nx+b_ny+c_nz$,
then from the recursion in Proposition \ref{prop.main.rec},
we obtain for $ n \geq 2$
\begin{eqnarray*}
	a_n &=& a_{n-1} + c_{n-1} + 2 a_{n-1} + b_{n-2} + c_{n-2} +f_{n-1} (f_n + f_{n-2}) \\
	b_n & = & f_n f_{n-1} \\
	c_n &=& a_{n-1}+ a_{n-2} + b_{n-2} + c_{n-2} ~.
\end{eqnarray*}
Eliminating $b_n$, this is equivalent to the system
\begin{eqnarray}\label{system1}
	a_n &=& a_{n-1} + 2 a_{n-1} + c_{n-1}+c_{n-2} +
	f_{n-2}f_{n-3} + f_{n-1} f_{n-2} + f_n f_{n-1} \\ \nonumber
	c_n &=& a_{n-1}+ a_{n-2} + c_{n-2} + f_{n-2}f_{n-3} ~.
\end{eqnarray}
Let $A(t), B(t), C(t)$ be the 
generating functions of the sequences $a_n, b_n , c_n$, ($ n \geq 2$), respectively.
We already know that (\cite[A001654]{oeis})
\begin{equation}\label{GFB}
	B(t) = \sum_{n \geq 2} f_n f_{n-1} t^n = \frac{t^2}{(1+t) (1-3t+t^2)} ~.
\end{equation}
From \eqref{system1} we obtain
\begin{eqnarray}\label{system2}
	A(t) & = & (t + 2 t^2) A(t) + (t +t^2) C(t) + (1 + t + t^2 ) B(t) \\ \nonumber
	C(t) &=& (t + t^2) A(t) + t^2 C (t) + t^2 B(t) ~.
\end{eqnarray}
Solving the system  of equations \eqref{system2} and 
using \eqref{GFB} we calculate
\begin{eqnarray}\label{system3}
	A(t) & = & \frac{t^2} {(1+t)^2 (1-3t+t^2)^2} ~,\\
	C(t) &= & \frac{t^3 + 2 t^4 - t^5}{(1+t)^2 (1-3t+t^2)^2 } \nn ~.
\end{eqnarray}
Since $\Mo (\Gamma_n) = M_n(1,1,1) = a_n + b _n + c_n$, adding the generating functions 
$A(t), B(t), C(t)$ we obtain
\begin{equation}\label{MoGF}
	\sum_{n \geq 2 } \Mo (\Gamma_n) t^n = \frac{(2-t)t^2}{(1+t)^2 (1-3t+t^2)^2} ~.
\end{equation}
Using partial fractions decomposition in \eqref{MoGF} and the expansions
\begin{eqnarray}\label{expansion1}
	\frac{1}{1-3t + t^2} & =& \sum_{n\geq 0} f_{2n+2} t^n ~,\\ \label{expansion2}
	\frac{1}{(1-3t + t^2)^2} & =& \sum_{n\geq 0} \frac{1}{5} \big( 
	(4n+2) f_{2n+2} + (3n+3) f_{2n+1} \big) t^n
\end{eqnarray}
we obtain
$$
\Mo ( \Gamma_n) = 
\frac{1}{25}
\big(
(3n+2) (-1)^n + (4n-5) f_{2n+2} + (3n+3) f_{2n+1} - (4n-3) f_{2n} - 3n f_{2n-1} 
\big)~,
$$
which can be simplified to 
the closed form expression for 	$\Mo ( \Gamma_n)$ in Theorem \ref{thm.main3}.
This is another way of writing the 
sum given in Theorem \ref{thm.main}.
\begin{theorem}\label{thm.main3}
	The Mostar index of Fibonacci cube $\Gamma_n$ is 
\begin{equation}\label{formula3}
	\Mo (\Gamma_n) = \frac{1}{25} \big((3n-2)f_{2n+2} + n f_{2n+1} + (3n+2)(-1)^n \big)~.
\nn\end{equation}
\end{theorem}


\section{The Wiener index and remarks}
In \cite{Wiener.Fibo} it is shown that
\begin{equation}\label{wiener}
W(\Gamma_n)=\sum_{k=1}^{n} f_{k}f_{k+1}f_{n-k+1}f_{n-k+2}~
\end{equation}
and that this sum can be evaluated as 
\begin{equation}\label{klavzar}
W(\Gamma_n)= \frac{1}{25} \big(
4 (n+1) f_n^2 +(9n+2) f_n f_{n+1} +6n f_{n+1}^2 \big) ~.
\end{equation}
In view of our formula \eqref{formula1} of Theorem \ref{thm.main} and \eqref{wiener} 
this means that 
$$
W(\Gamma_n)= \Mo (\Gamma_n) + \sum_{k=1}^{n} (f_{k}f_{n-k+1})^2 ~.
$$
The sum above is the sequence \cite[A136429]{oeis} with generating function
$$
\frac{t(1-t)^2}{(1+t)^2 (1-3t+t^2)^2}~.
$$
Adding the generating function \eqref{MoGF} to this, we get
\begin{equation}\label{GFW}
\sum_{n\geq1} W(\Gamma_n) t^n = \frac{t}{(1+t)^2 ( 1 - 3t + t^2)^2} ~.
\end{equation}
Using partial fractions and the expansions \eqref{expansion1} and \eqref{expansion2},
$W(\Gamma_n)$ ($n \geq 2$) is found to be 
$$
W(\Gamma_n) = 
\frac{1}{25} \big( (3n+2)f_{2n+3} + (n-2) f_{2n+2} -(n+2)(-1)^n\big)
$$
which is a somewhat simpler expression than \eqref{klavzar}.\\

It is also curious that in view of their generating functions \eqref{system3} and 
\eqref{GFW} which differ only by  factor of $t$, we have
$$a_n = M_n(1,0,0) =  W(\Gamma_{n-1})~.$$


\begin{thebibliography}{}
	\bibitem{Albertson} M. O. Albertson, The irregularity of a graph, Ars Combin. 46 (1997) 219--225.
	
	\bibitem{Klavzar_irr} Y. Alizadeh, E. Deutsch,  S. Klav\v{z}ar, On the irregularity of $\pi$-permutation graphs, Fibonacci cubes, and trees, Bull. Malays. Math. Sci. Soc. (2020). \url{https://doi.org/10.1007/s40840-020-00932-9}
		
	\bibitem{Mostar2} Y. Alizadeh, K. Xu, S. Klav\v{z}ar On the Mostar index of trees and product graphs, preprint \url{https://www.fmf.uni-lj.si/~klavzar/preprints/Mostar%20index%20(Nov%201%202020).pdf} 
	
	\bibitem{Mostar1} T. Do\v{s}li\'{c}, I. Martinjak, R. \v{S}krekovski, S.Tipuri\'{c} Spu\v{z}evi\'{c}, I. Zubac, Mostar index, J. Math. Chem. 56 (2018) 2995--3013.

	\bibitem{FiboLucas_irr}  \"O. E\u{g}ecio\u{g}lu, E. Sayg{\i}, Z. Sayg{\i}, The irregularity polynomials of Fibonacci and Lucas cubes, Bull. Malays. Math. Sci. Soc. (2020). \url{https://doi.org/10.1007/s40840-020-00981-0}
		
	\bibitem{Mostar3} F. Gao, K. Xu, T. Do\v{s}li\'{c}, On the difference of Mostar index and irregularity of graphs, Bull. Malays. Math. Sci. Soc. (2020). \url{https://doi.org/10.1007/s40840-020-00991-y}
	
	
	\bibitem{Hsu} W.-J. Hsu, Fibonacci cubes--a new interconnection technology, IEEE Trans. Parallel Distrib. Syst. 4 (1993) 3--12.
		
	\bibitem{survey} S. Klav\v{z}ar, Structure of Fibonacci cubes: a survey, J. Comb. Optim. 25 (2013) 505--522.
	
	\bibitem{Wiener.Fibo} S. Klav\v{z}ar, M. Mollard, Wiener index and Hosoya polynomial of Fibonacci and Lucas cubes, MATCH Commun. Math. Comput. Chem. 68 (2012) 311--324.
		
	\bibitem{Lucas} E. Munarini, C. P. Cippo,  N. Zagaglia Salvi, On the Lucas cubes, Fibonacci Quart. 39 (2001) 12--21 .
		
	\bibitem{oeis} OEIS Foundation Inc. (2021), The On-Line Encyclopedia of Integer Sequences, \url{http://oeis.org/A001654} and \url{http://oeis.org/A136429}
	
\end{thebibliography}
\end{document}